\documentclass[journal]{IEEEtran}

\ifCLASSINFOpdf
  \else  
\fi

\date{}
\usepackage{amsmath}
\usepackage{amsthm}
\usepackage{graphics} 
\usepackage{epsfig}
\usepackage{amsfonts} 
\usepackage{mathrsfs}
\usepackage{dsfont}
\usepackage{amssymb}
\usepackage{cite}
\usepackage{algorithm}
\usepackage{subfigure}
\usepackage{caption2}
\usepackage{multirow}
\usepackage{epstopdf}
\usepackage{verbatim}
\usepackage{color}
\usepackage[left=1.5cm,right=1.5cm,top=1.8cm,bottom=1.8cm]{geometry}
\begin{document} 

\title{ Reinforcement Learning-Based Optimal Control for Multiplicative-Noise Systems with Input Delay}
\author{Hongxia Wang, Fuyu Zhao, Zhaorong Zhang, Juanjuan Xu, Xun Li}
\author{Hongxia~Wang,~\IEEEmembership{}
       Fuyu~Zhao,~\IEEEmembership{}
        Zhaorong~Zhang,~\IEEEmembership{}
Juanjuan~Xu~\IEEEmembership{}
         and~Xun ~Li~\IEEEmembership {}
\thanks{Hongxia Wang and Fuyu Zhao are with the School
of Electrical and Automation Engineering, Shandong University of Science and Technology, Qingdao 30332, China, (e-mail: whx1123@126.com; 503171379@qq.com).}
\thanks{Zhaorong Zhang and Xun Li are with the Department of Applied Mathematics, The Hong Kong Polytechnic University
Hong Kong, China (e-mail: zhaorong.zhang@polyu.edu.hk; li.xun@polyu.edu.hk).}
\thanks{Juanjuan Xu is with Shandong University, Jinan 250061, China, (e-mail: juanjuanxu@sdu.edu.cn).}
}
\date{}
\maketitle
\newcommand{\RNum}[1]{\uppercase\expandafter{\romannumeral #1\relax}}
\newtheorem{corollary}{Corollary}
\newtheorem{assumption}{Assumption}
\newtheorem{remark}{Remark}
\newtheorem{lemma}{Lemma}
\newtheorem{theorem}{Theorem}
\newtheorem{definition}{Definition}
\newtheorem{problem}{Problem}
\newcommand{\no}{\nonumber}
\begin{abstract}
In this paper, the reinforcement learning (RL)-based optimal control problem is studied for multiplicative-noise systems, where input delay is involved and partial system dynamics is unknown. To solve a variant of Riccati-ZXL equations, which is a counterpart of standard Riccati equation and determines the optimal controller, we first develop a necessary and sufficient stabilizing condition in form of several Lyapunov-type equations, a parallelism of the classical Lyapunov theory. Based on the condition, we provide an offline and convergent algorithm for the variant of Riccati-ZXL equations. According to the convergent algorithm, we propose a RL-based optimal control design approach for solving linear quadratic regulation problem with partially unknown system dynamics. Finally, a numerical example is used to evaluate the proposed algorithm.


\end{abstract}
\section{Introduction}
The control based on reinforcement learning \cite{suttonmitpress1998} has received  paramount  attention because of its successful applications in games and simulators \cite{Silver2018science, Mnih2015nature}.  An increasing research effort is made on various RL algorithms for complex dynamical systems. The linear quadratic regulation (LQR) problem has reemerged as an important theoretical benchmark for RL-based control of complex systems with continuous-time state and action spaces.

Among RL-based control design for the LQR problem,  most work is for deterministic or additive noise systems, see  \cite{modarestac2014,bahareauto2014, bianauto2014, bianauto2016,lewis2011smc,jiang2018auto } and references  therein.  Multiplicative noise system explicitly incorporates model uncertainty and inherent stochasticity, and  is of benefit to robustness improvement of the controller.  Thus, there has also emerged some research for multiplicative noise systems \cite{jiang2011nn, xuauto2012, biantac2016,leongauto2020,wangneuro2016, gravell2021tac,coppens2020ldc,litac2022}. 

It should be stressed that time delay is seldom considered in RL-based control of the LQR problem for multiplicative noise systems  even though the model-based control design  for time delay systems has ever been fully investigated \cite{zhang2015dis}. Several RL algorithms are developed for solving optimal control problems of deterministic systems in presence of time delay \cite{zhang2011tnn, song2010neur, wei2010acta, zhang2014neur}.  Within the radius of our knowledge, it seems hard to generalize them to deal with LQR problem for multiplicative noise systems because these algorithms are problem-oriented. \cite{song2010neur} considers a particular nonlinear performance index, which does not include quadratic form index of the LQR  problem as a special case. A quasi-linear relation of the control input is assumed in \cite{wei2010acta},  and \cite{zhang2014neur} requires that  the underlying system can be converted into another delay-free system with the same dimension equivalently,  which seems to be somewhat strict for a general multiplicative-noise system. 
Two Q-learning techniques are proposed  for network control system with random delay and input-dependent noise, where the state augmentation is adopted and the original system is converted into a delay-free and high-dimensional system \cite{xuauto2012}. Given that the state space expansion may cause a large increase in learning time and memory requirements \cite{eric2010irs},  meanwhile,  the selection of exploration noise is not a trivial work for general RL problems, especially  for high-dimensional systems \cite{jiang2018auto}, a direct RL-based control design (avoiding augmentation) is provided for the optimal control involving input delay and input-dependent noise \cite{wangauto2022}. The design heavily depends on the special structure of systems. Therefore, there lacks RL-based control design for solving the general optimal control of systems with time delay and multiplicative noise. 


The problem is very involved even though the system dynamics is completely known. As shown in \cite{zhang2015dis}, different from the delay-free case,  the solvability condition and optimal controller of the problem are determined by Riccati-ZXL equations below,
\begin{align}
Z=&A'ZA+\bar A'X\bar A+Q-M'\Upsilon^{-1}M,\label{zequ}\\
X=&Z+\sum_{i=0}^{d-1}(A')^{i}M'\Upsilon^{-1}MA^i\label{xequ}
\end{align}
with
\begin{align}
\Upsilon=&R+B'XB+\bar B'Z\bar B,\\
M=&B'XA+\bar B'Z\bar A.
\end{align}
where $Z$ and $X$ are unknown matrices, and other matrices are known. Note that Riccati-ZXL equations or their variants in \cite{zhang2015dis} are not only nonlinear in $Z$ and $X$ but also coupled with each other. It is thus hard to attain the optimal control by solving them. Also, it is difficult to develop good parallel versions of the Newton's iterative method for solving Riccati-ZXL equations when there lacks a necessary and sufficient stabilizing condition for the multiplicative noise systems with input delay. More precisely, to obtain an approximate solution of the variants of Riccati-ZXL equations, it is necessary to develop a necessary and sufficient stabilizing condition similar to the classical Lyapunov theorem.

The goal of this paper is to approximately solve
optimal control  for general systems with input delay and multiplicative noise. The contribution of
this paper is multifold.  Firstly, we find a necessary and sufficient stabilizing condition of the general multiplicative noise systems with input delay. The condition generalizes the classical Lyapunov theorem and characterizes all predictor-feedback controllers. Secondly, we provide the recursively approximate solutions to the variant of  Riccati-ZXL equations and prove their convergence. Thirdly, we propose a novel RL method for  optimal control with input delay in stochastic setting. 

The remainder of the paper is organized as follows. Section II is devoted to deriving the necessary and sufficient stabilizing condition for the predictor-feedback.  As a application,  Section III gives two algorithms for solving the LQR for input-delay multiplicative-noise systems.  Numerical example is performed in Section IV.  Some conclusions are made in Section V. 

Notation: ${\cal R}^n$ stands for the $n$ dimensional Euclidean space; $I$ denotes the unit matrix; The superscript $ '$ represents the matrix transpose; For matrix $M$,  $M > 0$ (reps. $\ge 0$) means that it is positive definite (reps. positive semi-definite),  $M^i$ and  $M^{(i)}$ stand for a matrix with supscript $i$ and the power of matrix $M$;   For all matrices $A$ and $B$, $\mathrm{diag}\{A,B\}$ represents a block diagonal matrix with diagonal blocks $A$ and $B$. 
For matrix $D=(d_{ij})\in \mathbb{R}^{n \times m}$ and vector $x \in \mathbb{R}^{n}$, $||x||_D\dot =x'Dx$;
$\mathrm{vec}(D)=[d_{11}, \cdots, d_{1m}, d_{21}, d_{22},\cdots, d_{nm-1},d_{nm}]'$,
$ \underline{\mathrm{vec}}(D)=[d_{11}, \cdots, d_{1m}, d_{22}, d_{23},\cdots, d_{n-1m}, d_{mm}]'$,
$\mathrm{mat}(x)=xx'$;
$(\Omega,  {\cal F}, \{{\cal F}_k\}_{k\ge 0}, {\cal P})$ denotes a complete probability space.  $\{w_k\}_{k\ge 0}$, defined on this space,  is a white noise scalar  valued  sequence with zero mean and satisfies
${\mathsf E}[w_kw_s]=\delta_{ks}$, where $\delta_{ks}$ is the Kronecker function.  $\Omega$ is the sample space, ${\cal F}$ is a $\sigma$-field, $ \{{\cal F}_k\}_{k\ge 0}$ is the natural filtration generated by $\{w_k\}_{k\ge 0}$, and ${\cal P}$ is a probability measure \cite{yong1999stochastic} ; $x_k|_m={\mathsf E}[x_k| {\cal F}_m]$ denotes the conditional expectation of $x_k$ with respect to ${\cal F}_m$ and $x_k|_m^l=x_k|_l-x_k|_m$. A  stochastic process $X(w, k)$ is said to be ${\cal F}_k$-measurable if the map $w \rightarrow X(w, k)$ is measurable. Hence, $x_k|_m$ is ${\cal F}_m$-measurable\cite{yong1999stochastic}. 

\section{Problem statement and preliminaries}

\subsection{Problem Statement}
Consider the multiplicative-noise system below
\begin{align}
x_{k+1}=A_kx_k+B_ku_{k-d},\label{sys}
\end{align}
where $x_k \in {\cal R}^n$ is the system state, $u_k \in {\cal R}^m$ is the control input, $d$ is a positive integer and stands for the length of  time delay,  $\{w_k\}$ is a scalar white-noise process with zero mean and ${\mathsf E}[w_k'w_s]=\delta_{ks}$, and $\delta_{ks}$ is a Kronecker operator, $A_k=A+w_k\bar A$, $B_k=B+w_k\bar B$, $A$ and $ B$ are given constant matrices, and $\bar A$ and $\bar B$ are unknown constant matrices.

\begin{remark}
In system \eqref{sys}, $w_k(\bar Ax_k+\bar B_ku_{k-d})$ is used to represent the lumped disturbance of physical system, possibly including parameter variations and unmodeled inherent stochasity. Hence, it is hard to obtain exact $\bar A$ and $\bar B$ in practice.
\end{remark}

The performance index to be optimized  is given as 
\begin{align}
&J\dot={\mathsf E}\sum_{k=0}^\infty (x_k'Qx_k+u_{k-d}'Ru_{k-d}), \label{costfunj}
\end{align}
where $Q \ge 0$, $R>0$ and $(A, \bar A|Q^{1/2})$ is exactly observable. To guarantee well-posedness of the infinite-horizon control problem, the admissible controller are restricted to be  mean-square stabilizing and ${\cal F}_{k-d-1}$-measurable.  

We are interested in finding  a predictor-feedback controller $u_{k-d}$ which stabilizes system \eqref{sys} in mean-square sense and minimizes  $J$ in \eqref{costfunj}. 

The definitions of the stabilizability  under predictor-feedback controller and exact observability are put forward in the following. 
 \begin{definition}
System \eqref{sys} is said to be stabilizable if there exists a predictor-feedback controller  $u_{k-d}=-K$$x_{k|k-d-1}$, such that for any initial data $x_0, u_{-d},\cdots, u_{-1}$, the closed-loop system
\begin{align}
x_{k+1}=A_kx_k-B_kKx_{k|k-d-1} \label{ksys}
\end{align}
is asymptotically mean-square stable, that is, $\lim_{k\rightarrow +\infty}{\mathsf E}[x_k'x_k]=0$, where $K$ is a constant matrix. In this case, we also say that $K$ is stabilizing for short. 
\end{definition}



%
\begin{definition}
The multiplicative-noise system 
\begin{align}
x_{k+1}=f(x_k,w_k), 
y_k=Q^{1/2}x_k
\end{align}
is said to be exactly observable if for any $N \ge j$, 
\begin{align}
y_k\equiv 0, a.s. \forall j \le k \le N \Rightarrow x_j=0.
\end{align}
In particular,  if  both systems
\begin{align}
x_{k+1}=A_kx_k+B_ku_k, 
y_k=Q^{1/2}x_k
\end{align}
and 
\begin{align}
x_{k+1}=A_kx_k-B_kKx_k|_{k-d-1}, 
y_k=Q^{1/2}x_k
\end{align}
are exactly observable, it is also said that $(A,\bar A | Q^{1/2})$ and $(A-BK,\bar A-\bar B K |Q^{1/2})$  are exactly observable for short, respectively.
\end{definition}
\subsection{Optimal Solution of Multiplicaitve-Noise LQR with Input Delay and Exactly Known System Dynamics}
In the case that $A, B, \bar A$ and $\bar B$ are exactly known,  the analytic solution of $\min_u J$  subject to \eqref{sys} has been provided in  \cite[Th. 3]{zhang2015dis}, from which our control policy will be developed.  For ease of reading,  we  restate  \cite[Th. 3]{zhang2015dis} as a lemma.
\begin{lemma}\label{lem1}
Suppose that $(A, \bar A, Q^{1/2})$ is exactly observable. The problem  $\min_u J$ subject to \eqref{sys}  is uniquely solvable if and only if the coupled equations  below
\begin{align}
&\mathbf P^1=A'\mathbf P^1A+A'\mathbf P^dA+Q, \label{cmp1}
\\&\mathbf P^2=- M'\Upsilon^{-1} M,\label{cmp2}\\
&\mathbf P^{i}=A'\mathbf P^{i-1}A, i=3,\cdots, d+1,\label{cmpi}\\
&\Upsilon=R+\sum_{i=1}^{d+1}B'\mathbf P^{i}B+\bar{B}'\mathbf P^1\bar{B}>0, \label{upsi}\\
&M=\sum_{i=1}^{d+1}B'\mathbf P^{i}A+\bar B'\mathbf P^{1}\bar A \label{mm}
\end{align}
have a unique solution such that $\sum_{i=1}^{d+1} \mathbf P^i>0$. Moreover, for $k\ge d$,  the stabilzing and optimal controller is given by 
$u_{k-d}=-\Upsilon^{-1}Mx_k|_{k-d-1}$, and the optimal value function is $V_k={\mathsf E}[x_k'(\mathbf P^1x_k+\sum_{i=2}^{d+1}\mathbf P^{i}x_k|_{k-d+i-3})]$.
\end{lemma}

Equations \eqref{cmp1}-\eqref{cmpi} are a variant of Riccati-ZXL equations \eqref{zequ}-\eqref{xequ}. Note that equations \eqref{cmp1}-\eqref{cmpi} are also coupled with each other and nonlinear in $\mathbf P^i$ for $i=1,\cdots, d+1$. It is not easy to directly resolve \eqref{cmp1}-\eqref{cmpi} for $\mathbf P^i, i=1,\cdots, d+1$. Thus, it is necessary to develop some efficient algorithms to attain numerically approximate solution of  \eqref{cmp1}-\eqref{cmpi}.  For this, we rewrite the above lemma as follows. 

\begin{lemma}\label{lqrsol}
Suppose that $(A, \bar A, Q^{1/2})$ is exactly observable. The problem $\min_u J$  subject to \eqref{sys} is uniquely solvable if and only if 
Riccati-type equations 
\begin{align}
&P^{i-1}=A'P^{i}A+Q, i=1,\cdots, d-1,\label{ricpi}\\
&P^d=(A-BK)'P^d(A-BK)+(\bar A-\bar B K)'P^0(\bar A-\bar BK)\no\\&~~~~~~~~~~~~~~~~~~~~~~~~~~~~~~~~~+K'RK+Q,\label{ricpd}\\
&K=(R+B'P^dB+\bar B'P^0\bar B)^{-1}(B'P^dA+\bar B'P^0\bar A)\label{excgain}
\end{align}
 have a unique positive definite solution $P^i, i=0,\cdots, d$. Moreover, 
the optimal controller and the value function for $k>d$ are given by
$u_{k-d}=-Kx_k|_{k-d-1}$ and $V_k={\mathsf E}[x_k'(P^dx_k|_{k-d-1}+\sum_{i=1}^{d} P^{i-1}x_k|^{k-i}_{k-i-1})]$, respectively.
\end{lemma}

\begin{proof}
According to Lemma \ref{lem1}, we only need to show that the necessary and sufficient conditions in Lemma \ref{lem1} and this lemma are equivalent. First, we will derive the condition in this lemma from that in lemma \ref{lem1}. Denote  
\begin{align}
P^0=\mathbf P^1,
 P^i=P^{i-1}+\mathbf P^{d+2-i}, i=1,\cdots, d. \label{pipi}
\end{align}
Now direct algebraic manipulation based on \eqref{cmp1}-\eqref{cmpi} shows  that $P^i$ defined by \eqref{pipi} satisfies \eqref{ricpi}-\eqref{ricpd}. 
We then testify that $P^i$, $i=0,\cdots, d$, is positive definite. The positive definiteness of matrices  $\sum_{i=1}^{d+1}\mathbf P^i$ and $\Upsilon=R+\sum_{i=1}^{d+1}B'\mathbf P^{i}B+\bar{B}'\mathbf P^1\bar{B}$ in Lemma \ref{lem1} implies that $\mathbf P^1>0$ and $\mathbf P^i\le 0$, $i=2,\cdots, d+1$. In this case, \eqref{pipi} means $P^i\le P^{i-1}$, $i=1,\cdots, d$. In fact, it is easy to derive  from \eqref{pipi} that 
 $P^d=\sum_{j=1}^{d+1}\mathbf P^{i}$, and thus $P^d>0$.  Further, $0<P^d\le P^{d-1}\le \cdots \le P^0$.
In reverse, we shall demonstrate that the sufficient and necessary condition in this lemma implies that in Lemma \ref{lem1}. Note  that the  linear transformation \eqref{pipi} is nonsingular. Let
\begin{align}
\mathbf P^1=P^0,
\mathbf P^{d+2-i}= P^i-P^{i-1}, i=1,\cdots, d.
\end{align} 
It is directly deduced from\eqref{ricpi}-\eqref{excgain} that $\mathbf P^i$, $i=1,\cdots,d+1$, admits \eqref{cmp1}-\eqref{cmpi} with $\Upsilon$ and $M$ as in \eqref{upsi} and \eqref{mm}, respectively. As $P^i>0$, $i=0,\cdots, d$,  it is clear that $\sum_{i=1}^{d+1}\mathbf P^i=P^d>0$ and $\Upsilon=R+\sum_{i=1}^{d+1}B'\mathbf P^{i}B+\bar{B}'\mathbf P^1\bar{B}>0$.

\end{proof}

\subsection{Sufficient Stabilizing Condition}
Note that the optimal and stabilizing controller of $\min _u J$ subject to \eqref{sys} is in form of predictor-feedback. For proposing reasonable a RL-based control policy,   this subsection is devoted to characterizing all predictor-feedback controllers stabilizing system \eqref{sys}. 

\begin{lemma}\label{sufcon}
For given $K$ and $Q\ge 0$, assume $(A-BK,\bar A-\bar B K |Q^{1/2})$  is exactly observable.  If there exists matrix $P^i>0$, $i=0,\cdots, d$, satisfying the following equations 
\begin{align}
&P^{i-1}=A'P^{i}A+\bar A'P^0\bar A+Q, i=1,\cdots, d-1,\label{lyapi}\\
&P^d=(A-BK)'P^d(A-BK)\no\\&~~~~~~~~~~~~~~+(\bar A-\bar BK)'P^0(\bar A-\bar BK)+Q\label{lyapd},
\end{align}
then system \eqref{ksys} is asymptotically mean-square stable.   
\end{lemma}

\begin{proof}
Our proof is based on Lyapunov stability theorem.
Define a Lyapunov functional candidate
\begin{align}
V_{k}=&{\mathsf E}[x_{k}'(P^dx_{k}|_{k-d-1}+\sum_{i=0}^dP^{i-1}x_{k}|_{k-1-i}^{k-i})], \label{lyafuncan}
\end{align}
where $P^i, i=0,\cdots, d$, is the positive definite solution to equations \eqref{lyapi}-\eqref{lyapd}, $x_{k}|_{k-1-i}^{k-i}=x_{k}|_{k-i}-x_{k}|_{k-1-i}$, and 
\begin{align}
&x_{k+1}|_{k-i}=Ax_k|_{k-i}-BKx_{k|k-d-1}, i=1,\cdots,d. \label{kisys}
\end{align}
which is obtained by taking conditional expectations over  ${\cal F}_{k-i-1}$ on both sides of the system \eqref{ksys}. In view of \eqref{kisys},  there hold
\begin{align}
&x_{k+1}|_{k+1-i}-x_{k+1}|_{k-i}=A(x_k|_{k+1-i}-x_k|_{k-i}),\no\\&~~~~~~~~~~~~~~~~~~~~i=2,\cdots,d-1, \label{tildexi}\\
&x_{k+1}|_{k}-x_{k+1}|_{k-1}=w_k(\bar Ax_k-\bar B Kx_k|_{k-d-1}).\label{tildex1}
\end{align}
Along with system \eqref{ksys},  \eqref{tildexi} and \eqref{tildex1},  $V_{k+1}$ is rewritten as below.
\begin{align}
&V_{k+1}={\mathsf E}[||x_{k+1}|_{k-d}||+\sum_{i=0}^d||x_{k+1}|_{k-i}^{k+1-i}||_{P^{i-1}}]\no\\
=&{\mathsf E}[ ||Ax_k|_{k-d-1}^{k-d}+(A-BK)x_{k|k-d-1})||_{P^d}\no\\
&+\sum_{i=2}^d ||x_k|_{k-i}^{k+1-i}||_{A'P^{i-1}A}\no\\
&+||\bar A-\bar B K)x_k|_{k-d-1}+\bar Ax_k|_{k-d-1}^{k-1}||_{P^0}\no\\
=&{\mathsf E}||x_k|_{k-d-1}^{k-d}||_{A'P^dA+\bar A'P^0\bar A}\no\\
&+||x_k|_{k-d-1}||_{(A-BK)'P^d(A-BK)+(\bar A-\bar BK)'P^0(\bar A-\bar BK)}\no\\
&+\sum_{i=1}^{d-1} ||x_k|_{k-i-1}^{k-i}||_{A'P^{i}A+\bar A'P^0\bar A}.
\end{align}
Combining it with \eqref{lyapi}-\eqref{lyapd} shows
\begin{align}
V_{k+1}-V_k=-{\mathsf E}[x_k'Qx_k]\le 0. \label{deltavk}
\end{align}
The inequality above has used the positive semi-definiteness of  $Q$. If  ${\mathsf E}[x_k'Qx_k]=0$ for $k=j,\cdots, N$, where $N>0$ is arbitrary and $j$ is the initial time,  then $Q^{1/2}x_k \equiv 0$ holds for $k$ in $[j, N]$ almost surely. Recall  the exact observability of $(A-BK, \bar A-\bar BK|Q^{1/2})$.  In this case, $x_j=0$. Initilizing the system at any $k$,  $x_k=0$ for  $k=j,\cdots, $  almost surely. 
According to  Lyapunov stability theory,  system \eqref{ksys} is asymptotically mean-square stable. 
\end{proof}

\subsection{Necessary Stabilizing Condition}

We have provided a sufficient stabilizing condition for system \eqref{ksys} in form of  Lyapunov-type equations. We are also interested in discussing necessary stabilizing conditions of system \eqref{ksys}.  

\begin{lemma}\label{neccon}
For given $K$ and $Q\ge 0$, if  system \eqref{ksys} is asymptotically mean-square stable,  the following Lyapunov-type equations 
\begin{align}
&S^0=(\bar{A}-\bar{B}K)S^d(\bar{A}-\bar{B}K)'+\bar{A}\sum_{i=0}^{d-1}S^i\bar{A}',\label{ms0}\\
&S^i=AS^{i-1}A',\label{msi}\\
&S^d=(A-BK)S^d(A-BK)'+AS^{d-1}A'+Q\label{msd}
\end{align}
have a positive semi-definite solution, and matrix 
\begin{small}
$${\cal A}=\left[\begin{matrix}\bar{A}\otimes\bar{A}&\bar{A}\otimes\bar{A}&\bar{A}\otimes\bar{A}&\cdots&  (\bar{A}-\bar{B}K)\otimes (\bar{A}-\bar{B}K)\\
{A}\otimes{A}&0&0&\cdots&0\\
0&{A}\otimes{A}&0&\cdots&0\\
0&0&{A}\otimes{A}&\cdots&0\\
0&0&0&\cdots& (A-BK)\otimes (A-BK)\end{matrix}\right]$$
\end{small} is Schur.
\end{lemma}
\begin{proof}
Our proof depends on two important facts. Fact 1 is that $\lim_{k\rightarrow +\infty}{\mathsf E}[x_k'x_k]=0$ is equivalent to $\lim_{k\rightarrow +\infty}{\mathsf E}[x_kx_k']=0$. Fact 2 is  that $\lim_{k\rightarrow +\infty}{\mathsf E}[x_k'x_k]=0$ means $\lim_{k\rightarrow +\infty}{\mathsf E}[x_k|_{k-i}'x_k|_{k-i}]=0$ and $\lim_{k\rightarrow +\infty}{\mathsf E}[(x_k-x_k|_{k-i})'(x_k-x_k|_{k-i})]=0$ because  of  ${\mathsf E}[x_k'x_k]={\mathsf E}[x_k|_{k-i}'x_k|_{k-i}]+{\mathsf E}[(x_k-x_k|_{k-i})'(x_k-x_k|_{k-i})]$, ${\mathsf E}[x_k|_{k-i}'x_k|_{k-i}]\ge 0$ as well as  ${\mathsf E}[(x_k-x_k|_{k-i})'(x_k-x_k|_{k-i})]\ge $ for $0<i<k$.

Let $X_k^i={\mathsf E}[x_k|_{k-i-1}x_k|_{k-i-1}']$ for $i=0,\cdots, d$. It can be derived from the predictor system \eqref{kisys} that
\begin{small}
\begin{align}
X_{k+1}^i=&AX_k^{i-1}A'-BKX_k^dA'-{A}X_k^dK'B'\no\\&~~~~~~~~+{B}KX_k^dK'{B}', i=1,\cdots,d, \label{exxki}\\
X_{k+1}^0=&AX_k^0A'+\bar{A}X_k^0\bar{A}'+BKX_k^dK'B'+\bar{B}KX_k^dK'\bar{B}'\no\\&-AX_k^dK'B'-\bar{A}X_k^dK'\bar{B}'-BK{X}_k^dA'-\bar{B}K{X}_k^d\bar{A}'.\label{exxk0}
\end{align}
\end{small}
Denote $\Delta X_k^{i}=X_k^{i}-X_k^{i+1}$ for $i=0,\cdots, d-1$. \eqref{exxk0} means
\begin{small}
\begin{align}
\Delta X_{k+1}^0=&\bar{A} X_k^0\bar{A}'-\bar{A}X_k^dK'\bar{B}'-\bar BK{X}_k^d\bar A'+\bar{B}K{X}_k^dK'\bar{B}', \label{deltaxxk0}\\
\Delta X_{k+1}^i=&A\Delta X_k^{i-1}A', i=1,\cdots, d-1, \label{deltaxxki}\\
X_{k+1}^d=&A\Delta X_k^{d-1}A'+(A-BK)X_k^d(A-BK)'.\label{exxkd}
\end{align}
\end{small}
When system \eqref{ksys} is asymptotically mean-square stable, according to Fact 1 and 2, $\Delta X_k^i$, $i=0,\cdots, d-1$ and $X_k^d$ are also  asymptotically stable, which is equivalent to that matrix  ${\cal A}$ 
is Schur  from the vectorized systems of the deterministic systems \eqref{deltaxxk0}-\eqref{exxkd}.

Denote $X^i=\sum_{k=0}^\infty X_k^i$ for $i=0,\cdots, d$ and $X_0^0=\cdots=X_0^d=Q\ge 0$. In view of  Theorem 1 in \cite{huangajc2008}, the stabilization of
system \eqref{sys} guarantees the existence of $X^i$ for $i=0,\cdots, d$. Moreover, we have
$0 \le X^d \le \cdots \le X^0<\infty$. Then, it can be deduced from \eqref{exxki}-\eqref{exxk0} that
\begin{align}
X^i-Q=&AX^{i-1}A'-BKX^dA'-{A}X^dK'B'\no\\&~~~~~~~~+{B}KX^dK'{B}', i=1,\cdots,d, \label{xi}\\
X^0-Q=&AX^0A'+\bar{A}X^0\bar{A}'+BKX^dK'B'\no\\
&+\bar{B}KX^dK'\bar{B}'-AX^dK'B'-\bar{A}X^dK'\bar{B}'\no\\
&-BK{X}^dA'-\bar{B}K{X}^d\bar{A}'.\label{x0}
\end{align}
Let $S^i=X^i-X^{i+1}$ for $i=0,\cdots, d-1$ and $S^d=X^d$. Then $X^0=S^d+\sum_{i=0}^{d-1}S^i$. Now it follows from equalities \eqref{xi} and \eqref{x0} that \eqref{ms0}-\eqref{msd} hold.
Notice that $S^d=X^d=\sum_{k=0}^\infty X_k^d$ and $Q\ge 0$. It is easy to know $S^d \ge 0$. Similarly, $S^0=\sum_{k=0}^\infty (X_k^i-X_k^{i+1})$ and $X_k^i-X_k^{i+1}\ge 0$ result in $S^i\ge 0$ for $i=0,\cdots, d-1$.   
\end{proof}
\begin{remark}
In the case of  $d=0$,  the Lyapunov-type equations \eqref{ms0}-\eqref{msd} are reduced as
\begin{align}
S^d=&(A-BK)S^d(A-BK)'\no\\&+(\bar{A}-\bar{B}K)S^d(\bar{A}-\bar{B}K)'+Q,
\end{align}
which is a standard generalized  Lyapunov equation.
\end{remark}

\begin{remark}
In the case of $\bar A=0$, the Lyapunov-type equations  \eqref{ms0}-\eqref{msd} are reduced as  
\begin{align}
S^d&=(A-BK)S^d(A-BK)'+A^{(d)}\bar{B}KS^dK'\bar{B}'{A^{(d)'}}+Q, \label{ssd}
\end{align}
which is actually a standard generalized  Lyapunov equation related to the multiplicative-noise system 
\begin{align}
x_{k+1}=Ax_k+(B+A^{(d)}\bar Bw_k)u_k.
\end{align}
The generalized Lyapunov equation \eqref{ssd} is in accordance with  \cite[eq. (18)]{tanscl2019}.
\end{remark}

\subsection{The Dual Relation between  Lyapunov-Type Equations }
To show that  the sufficient condition proposed in Lemma \ref{sufcon} is also necessary, we will regard the right-hand sides of the Lyapunov-type equations \eqref{lyapi}-\eqref{lyapd} and \eqref{ms0}-\eqref{msd} (neglecting the constant terms ) as  linear operators  from ${\mathcal R}^{n(d+1)\times n(d+1)}$ to  ${\mathcal R}^{n(d+1)\times n(d+1)}$ and discuss the relation between these two operators, where ${\mathcal R}^{n(d+1)\times n(d+1)}$ denotes $n(d+1)\times n(d+1)$ real matrix space.

Let  $f$ and $g$ be  linear operators from  ${\mathcal R}^{n(d+1)\times n(d+1)}$ to  ${\mathcal R}^{n(d+1)\times n(d+1)}$ as below:
\begin{small}
\begin{align}
f(P)=&\mathrm{diag}\{\bar A'P_{0}\bar A+ A'P_{1} A,\cdots, \bar A'P_{0}\bar A+ A'P_{d} A, \no\\&(\bar A-\bar BK)'P_0(\bar A-\bar BK)+(A-BK)'P_d(A-BK)\},\\
g(M)=&\mathrm{diag}\{\sum_{k=0}^{d-1}\bar AM_0\bar A'+(\bar A-\bar BK)M_d(\bar A-\bar BK)', A'M_1A, \no\\&\cdots, A'M_{d-2}A,A'M_{d-1}A+(A-BK)M_d(A-BK)'\},
\end{align}
\end{small}
where $P=\left[\begin{matrix}P_0 &* & \cdots & *\\ *& P_1& \cdots & * \\ * & * & \cdots & * \\ * & * &\cdots&P_d\end{matrix}\right] \in {\mathcal R}^{n(d+1)\times n(d+1)}$, $M=\left[\begin{matrix}M_0 &* & \cdots & *\\ *& M_1& \cdots & * \\ * & * & \cdots & * \\ * & * &\cdots&M_d\end{matrix}\right]\in {\mathcal R}^{n(d+1)\times n(d+1)}$, and $*$ denotes any real matrix. 

\begin{lemma}\label{duallem}
The linear operators $f$ and $g$ are dual on Hilbert space $({\mathcal R}^{n(d+1)\times n(d+1)}, \langle \cdot, \cdot\rangle)$, where $\langle \cdot, \cdot\rangle$ stands for  inner product  and is defined by  trace of  matrix product(denoted by $\mathrm {Tr}$).
\end{lemma}
\begin{proof}
Denote $f^*$ as  dual operator of $f$. Then for any $P, M \in {\mathcal R}^{n(d+1)\times n(d+1)}$, there holds 
\begin{align}
\langle f(P), M\rangle=\langle P, f^*(M)\rangle.\label{dualope}
\end{align}
Notice that  
\begin{small}
\begin{align}
&\langle f(P), M\rangle=\mathrm{ Tr}(f(P)M)\no\\
=&\mathrm{ Tr}(\sum_{i=1}^{d}(\bar A'P_{0}\bar A+ A'P_{i} A)M_{i-1}+(\bar A-\bar BK)'P_0(\bar A-\bar BK)M_{d}\no\\&~~~~~~~~+(A-BK)'P_d(A-BK)M_{d})\no\\
=&\mathrm{ Tr}(\sum_{i=1}^{d}[P_0(\bar A'M_{i-1}\bar A)+P_i(A'M_{i-1}A)]+P_0(\bar A-\bar BK)M_{d}\no\\&~~~~~~~~\times(\bar A-\bar BK)'+P_d(A-BK)M_{d}(A-BK)')\no\\
=&\langle P, g(M)\rangle, 
\end{align}
\end{small}
which together with \eqref{dualope} means $f^*(M)=g(M)$.
The proof is completed.
\end{proof}

 The dual relation provides  theoretical basis for  the following lemma, which is a necessary condition of stabizabilition. 
 
\begin{lemma}\label{necsta}
For given $K$ and $Q\ge 0$, assume $(A-BK,\bar A-\bar B K |Q^{1/2})$  is exactly observable. The Lyapunov-type equations \eqref{lyapi}-\eqref{lyapd} have a unique positive definite  solution if  system \eqref{ksys} is asymptotically mean-square stable.
\end{lemma}
\begin{proof}
The proof will be divided into two parts.  One is  to show that \eqref{lyapi}-\eqref{lyapd} have a unique solution, the other is  to prove positive definiteness of  the unique solution.

First,  the dual relation in Lemma \ref{duallem} is intrinsic argument that \eqref{lyapi}-\eqref{lyapd} have a unique solution. Assume that system \eqref{ksys} is asymptotically mean-square stable. For ease of reading,  rewrite the equations  \eqref{lyapi}-\eqref{lyapd} as 
\begin{align}
\begin{bmatrix} \mathrm{vec} ({P}^{0})\\ \vdots\\ \mathrm{vec} (P^d) \end{bmatrix}={\cal A}'\begin{bmatrix}  \mathrm{vec} ({P}^{0})\\ \vdots\\ \mathrm{vec} (P^d)\end{bmatrix}+\begin{bmatrix} \mathrm{vec} ( Q)\\ \vdots\\ \mathrm{vec} ( Q)\end{bmatrix}. \label{varofric}
\end{align}
According to Lemma \ref{neccon}, matrix  ${\cal A}$ is Schur when system \eqref{ksys} is asymptotically mean-square stable, so is its transpose.  Now it is ready to see that \eqref{varofric} has a unique solution and thereby \eqref{lyapi}-\eqref{lyapd} have a unique solution.

Second, we will show positive definiteness  of the unique solution. Let $V_k$ be as in \eqref{lyafuncan} and $P^i$ admit \eqref{lyapi}-\eqref{lyapd}. 
From \eqref{deltavk}, we can get 
\begin{align}
\sum_{k=j}^{N} (V_{k}- V_{k+1})=V_j-V_{N+1}={\mathsf E}[\sum_{k=j}^{N} x_k'Qx_k].
\end{align}
Take limit on both sides of the above equality with respect to $N\rightarrow \infty$.
Since system \eqref{ksys} is asymptotically mean-square stable, $V_{N+1}\rightarrow 0$ as  $N\rightarrow \infty$. Consequently, 
\begin{align}
V_{j}={\mathsf E}[\sum_{k=j}^{\infty} x_k'Qx_k]\label{keysta0}
\end{align}
for any $j \ge d$.
 Let the initial state at time $j$ be $x_j=c$ and $x_j=w_s c, s=j-1, \cdots, j-d$, where $c \neq 0$ is an  arbitrary constant vector. Direct calculation gives $V_j=c'P^dc$ and $V_j= c'P^{i-1}c, i=1,\cdots, d$, respectively. From $Q\ge 0$, there also has that $V_j={\mathsf E}[\sum_{k=j}^{\infty} x_k'Qx_k] \ge 0$. Consequently, the positive semi-definiteness of $P^i \ge 0$ follows, where $i=0,\cdots, d$.  If $P^i$, $i=0,\cdots, d$, is not positive definite and $c\neq 0$ belongs to the kernal space of $P^i$ (i.e., $P^ic=0$), then for $\forall  j \le k\le N$ and any $N\ge j$, $y_k=Q^{1/2}x_k=0$ almost surely, which contradicts the exactly observability of system \eqref{ksys} with output equation $y_k=Q^{1/2}x_k$. Therefore,  $P^i>0, i=0,\cdots, d$. The proof is now completed. 
\end{proof}

\begin{remark}
From the above proof,   the exact observability serves  to guarantee that the positive semi-definite solution of the Lyapunov equations \eqref{lyapi}-\eqref{lyapd} is positive definite when $Q$ is positive semi-definite. In other words, if $Q>0$, the Lyapunov equations \eqref{lyapi}-\eqref{lyapd} still have a positive definite solution even though not assume the exact observability of $(A-BK,\bar A-\bar B K |Q^{1/2})$.
\end{remark}

It is noticed that the coupled Lyapunov-type equations \eqref{lyapi}-\eqref{lyapd} including  $d+1$ matrix equations actually can be reduced to a pair of coupled Lyapunov-type equations.
\begin{remark}\label{rsufcon}
For given $K$ and $Q$, the following Lyapunov equations  
\begin{small}
\begin{align}
P^{0}
=&{A^{(d)}}'P^dA^{(d)}+\sum_{k=0}^{d-1}{A^{(k)}}'\bar A'P^{0}\bar AA^{(k)}+\sum_{k=0}^{d-1}{A^{(k)}}'QA^{(k)},\label{rp0}\\
P^d=&(A-BK)'P^d(A-BK)+(\bar A-\bar BK)'P^0(\bar A-\bar BK)+Q\label{rpd}
\end{align}
\end{small}
have a solution $(P^0, P^d)$ if and only if \eqref{lyapi}-\eqref{lyapd} have a solution $P^i, i=0,\cdots, d$.
\end{remark}

The conclusion  in  this remark can be obtained by straightforward algebraic manipulation. If  \eqref{lyapi}-\eqref{lyapd} have a solution.
From \eqref{lyapi}, one can deduce
\begin{align}
&P^{i-1}=A'P^{i}A+\bar A'P^{0}\bar A+Q,\no\\
&={A^{(2)}}'P^{i+1}A^{(2)}+A'\bar A'P^{0}\bar AA+A'QA+\bar A'P^{0}\bar A+Q,\no\\
&={A^{(d-i+1)}}'P^dA^{(d-i+1)}\no\\
&+\sum_{k=0}^{d-i}A^{(k)'}\bar A'P^{0}\bar AA^{(k)}+\sum_{k=0}^{d-i}{A^{(k)}}'QA^{(k)}.\end{align}
Let $i=1$, then \eqref{rp0} appears. 
Plugging the above equality with $i=1$ into \eqref{lyapd} results in \eqref{rpd}. The sufficiency part is now evident.

If \eqref{rp0}-\eqref{rpd} has a solution $(P^0, P^d)$, then we can define $P^{i-1}$ by $P^{i-1}={A^{(d-i+1)'}}P^dA^{(d-i+1)}+\sum_{k=0}^{d-i}A^{(k)'}$ $\bar A'P^{0}\bar AA^{(k)}+\sum_{k=0}^{d-i}A^{(k)'}QA^{(k)}$ for $i=1,\cdots, d$. Obviously, such $P^i, i=0,\cdots, d$, admits Lyapunov-type equations \eqref{lyapi}-\eqref{lyapd}.


\section{Iterative optimal control design}
In this section, with the aid of stabilizing condition obtained in the proceeding section, we will propose two control designs for minimizing the performance index $ J$ in \eqref{costfunj} of the  multiplicative-noise system \eqref{sys}.

%
%
%
\subsection{Offline and Model-Based Algorithm}
From Lemma \ref{lem1}, it is not easy to get the optimal control by solving Riccati-type equations \eqref{cmp1}-\eqref{cmpi}. For this, we rewrite \eqref{cmp1}-\eqref{cmpi} as Riccati-type equations \eqref{ricpi}-\eqref{ricpd} so as to find the iterative solutions by virtue of Lyapunov-type equations \eqref{lyapi}-\eqref{lyapd} and analyze their convergence via the proposed stabilizing condition in Section 2.
%

The following theorem provides an offline and model-based optimal controller for the LQR $\min_u J$  in \eqref{costfunj} subject to \eqref{sys}. It approximates the solution to the Riccati-type  equations \eqref{ricpi}-\eqref{ricpd} via the solutions of a sequence of Lyapunov-type equations, which is also the theoretical basis of our data-driven algorithm.
 
\begin{theorem}\label{main}
For given $Q \ge 0$, assume $(A, \bar A|Q^{1/2})$ is exactly observable.  Let $K_0$ be stabilizing, and  $P_j^i$, $i=0,\cdots, d$,  the positive definite solution of the Lyapunov-type equations
\begin{align}
&P_j^{i-1}=A'P_j^{i}A+\bar A'P_j^0\bar A+Q, i=1,\cdots, d-1,\label{ipki}\\
&P_j^d=(A-BK_j)'P_j^d(A-BK_j)\no\\&~~~~~~~+(\bar A-\bar BK_j)'P_j^0(\bar A-\bar BK_j)+K_j'RK_j+Q\label{ipkd},
\end{align}
where  $K_j$, $j=1,2,\cdots$, is defined recursively by 
\begin{small}
\begin{align}
K_j=(R+B'P_{j-1}^dB+\bar{B}'P_{j-1}^0\bar{B})^{-1}(B'P_{j-1}^dA+\bar{B}'P_{j-1}^0\bar{A}). \label{ki}
\end{align}
\end{small}
Then, the following properties hold:
\begin{itemize}
\item[1)] system \eqref{sys} can be stabilized by $K_j$;
\item [2)]  $0< P^i_{j+1}\le P^i_{j}$ for $ i=0,\cdots, d$;
\item [3)]$ lim_{j \to \infty} P^i_{j}=P^{i} $ for $i=0,\cdots, d$,  $lim_{j\rightarrow \infty} K_j=K$,  where $P^i$ obeys \eqref{ricpi}-\eqref{ricpd}, and $K$ is as in \eqref{excgain}.
\end{itemize}
\end{theorem} 

\begin{proof}
It should be noticed a fact that if $(A, \bar A|Q^{1/2})$ is exactly observable, then for any matrices $K$, $R>0$ and $Q_1\ge 0$, $(A-BK, \bar A-\bar BK|(Q+K'RK+Q_1)^{1/2})$ is also exactly observable \cite{huang2006wcica}.  With this fact,
 Lemma \ref{sufcon} and \ref{necsta} can be used to show that system \eqref{sys} can be stabilized by $-K_jx_k|_{k-d-1}$ and the Lyapunov-type equations \eqref{ipki}-\eqref{ipkd} have  a unique positive definite solution, respectively.  What follows is the proof in details.  

We at first rewrite equation \eqref{ipkd} as 
\begin{small}
\begin{align}
P_j^d
=&(A-BK_{j+1})'P_j^d(A-BK_{j+1})\no\\&+(\bar A-\bar BK_{j+1})'P_j^0(\bar A-\bar BK_{j+1})+K_j'RK_j+Q\no\\
&+K_{j+1}'(AP_j^dB+\bar AP_j^0\bar B)+(AP_j^dB+\bar AP_j^0\bar B)'K_{j+1}\no\\&-K_{j+1}'(N_{j+1}-R)K_{j+1}-K_{j}'(AP_j^dB+\bar AP_j^0\bar B)\no\\
&-(AP_j^dB+\bar AP_j^0\bar B)'K_{j}+K_{j}'(N_{j+1}-R)K_{j}\no\\
=&(A-BK_{j+1})'P_j^d(A-BK_{j+1})\no\\&+(\bar A-\bar BK_{j+1})'P_j^0(\bar A-\bar BK_{j+1})+Q\no\\
&+2K_{j+1}'N_{j+1}K_{j+1}-K_{j+1}'(N_{j+1}-R)K_{j+1}\no\\
&-K_{j}'N_{j+1}K_{j+1}-K_{j+1}'N_{j+1}K_{j}+K_{j}'N_{j+1}K_{j}\no\\
=&(A-BK_{j+1})'P_j^d(A-BK_{j+1})\no\\&+(\bar A-\bar BK_{j+1})'P_j^0(\bar A-\bar BK_{j+1})+Q\no\\
&+(K_{j+1}-K_j)'N_{j+1}(K_{j+1}-K_j)+K_{j+1}'RK_{j+1}, \label{ripkd}
\end{align}
\end{small}
where $N_{j+1}=R+B'P_j^dB+\bar B'P_j^0\bar B$.

Let $\delta P_j^i=P_j^i-P_{j+1}^i$ for $i=0,\cdots, d$. By associating \eqref{ripkd} with Lyapunov-type equations \eqref{ipki}-\eqref{ipkd},  it can be obtained that
\begin{align}
&\delta P_j^{i-1}=A'\delta P_j^{i}A+\bar A'\delta P_j^0\bar A+Q, i=1,\cdots, d-1,\label{deltaripki}\\
&\delta P_j^d=(A-BK_{j+1})'\delta P_j^d(A-BK_{j+1})\no\\&~~~~~~~~+(\bar A-\bar BK_{j+1})'\delta P_j^0(\bar A-\bar BK_{j+1})\no\\&~~~~~~~~+(K_{j+1}-K_j)'N_{k+1}(K_{j+1}-K_j).\label{deltaripkd}
\end{align}

Subsequently, according to \eqref{ripkd} and \eqref{deltaripki}-\eqref{deltaripkd}, we shall show that $1)-2)$ hold.  

 In the case of $j=0$, since $K_0$ is stabilizing and $(A-BK_0, \bar A-\bar BK_0|(Q+K_0'RK_0)^{1/2})$ is exactly observable,  it follows from Lemma \ref{necsta} that Lyapunov-type equations \eqref{ipki}-\eqref{ipkd} have a unique positive definite solution $P^i_0, i=0,\cdots, d$. Further,  one can obtain that $(K_{1}-K_0)'N_{1}(K_{1}-K_0)\ge 0$ and $(A-BK_0, \bar A-\bar BK_0|(Q+(K_{1}-K_0)'N_{1}(K_{1}-K_0)+K_1'RK_1)^{1/2})$ is exactly observable.  According to Lyapunov-type equations \eqref{ipki} and \eqref{ripkd}(for $j=0$) and Lemma \ref{sufcon}, it is inferred that $K_1$ is stabilizing. Recall the exact observability of $(A-BK_1, \bar A-\bar BK_1|(Q+K_1'RK_1)^{1/2})$.  From Lemma \ref{necsta}, the Lyapunov-type equations \eqref{ipki}-\eqref{ipkd} with $j=1$ have a unique positive definite solution $P^i_1,i=0,\cdots, d$. 
Observe  the Lyapunov-type equations \eqref{deltaripki}-\eqref{deltaripkd} with $j=0$,  where $K_1$ is stabilizing and $ (K_{j+1}-K_j)'N_{j+1}(K_{j+1}-K_j)\ge 0$. Without the exact observability,  from the proof of  Lemma \ref{necsta}, it can be deduced that \eqref{deltaripki}-\eqref{deltaripkd} wtih $j=0$ have a positive semi-definite solution $\delta P_0^i, i=0, \cdots, d$, i.e., $P^i_0\ge P^i_1$, $i=0,\cdots, d$. 

Repeat the above process for $j\ge 1$. It is evident that the conclusions  $1)-2)$ in this theorem hold.

Finally,  the convergence of $P_j^i$ with respect to $j$ is  to be shown.  ii) implies that  for any $i=0,\cdots, d$, the matrix sequence $\{ P^i_j\}$ is bounded from below and decreases monotonically with respect to $j$. Thus,  for any $i=0,\cdots, d$, $\{ P^i_j\}$ is convergent as $j\rightarrow \infty$. Denote $\lim_{j\rightarrow \infty} P^i_j$ as $P^i$ for $i=0,\cdots, d$. Taking the limit with respect to $j$ on the both sides of \eqref{ipki}-\eqref{ki}, we obtain that  $P^i$ obeys the Riccati-type equations \eqref{ricpi}-\eqref{ricpd}, where 
 $\lim_{j\rightarrow \infty} K_j =K$. Moreover, for any $i=0,\cdots, d$,  the  positive definiteness of $P_j^i$ means $P^i>0$.

Until now, the proof of Theorem \ref{main} is completed.

\end{proof}

\begin{remark}
 \cite[Th. 1]{1971hewer} provides a numerical method for standard Riccati equation by iteratively solving a sequence of Lyapunov equations. Theorem 1 is a counterpart of \cite[Th. 1]{1971hewer} because it iteratively solves the variant of Riccati-ZXL equations, which determines the optimal solution of the LQR problem for multiplicative-noise systems with input delay. 
\end{remark}

\subsection{Online Algorithm for Multiplicative-Noise LQR with Input Delay and Partial Unknown Dynamics}

We turn to find an online algorithm for solving $\min_u J$  in \eqref{costfunj} subject to \eqref{sys} with unknown system dynamics $\bar A$ and $\bar B$ and exactly observable $(A, \bar A|Q^{1/2})$.


For any $k\ge d$,  define $\bar V_{k}$ as
\begin{align}
\bar V_k={\mathsf E}[||x_{k}|_{k-d-1}||_{P_j^d}+\sum_{i=1}^d||x_{k}|_{k-i-1}^{k-i}||_{P_j^{i-1}}], \label{rvaluefunc}
\end{align}
where $P_j^i$ for $i=0,\cdots, d+1$ admits \eqref{ipki}-\eqref{ipkd} with $k=j$. 

Rewrite system \eqref{sys} as 
\begin{align}
x_{k+1}=&A_kx_k|_{k-d-1}^{k-1}+(A_k-BK_j)x_k|_{k-d-1}\no
\\&+B_k(u_{k-d}+K_jx_k|_{k-d-1}), \label{rksys}
\end{align}
where $K_j$ is as in \eqref{ki}.

It follows from \eqref{rvaluefunc} and \eqref{rksys} that
\begin{align}
&\bar V_k-\bar V_{k+1}\no\\
=&{\mathsf E}[\sum_{i=1}^d||x_{k}|_{k-i-1}^{k-i}||_{P_j^{i-1}-A'P_j^{i}A-\bar A'P_j^0\bar A}\no\\
&-||(A-BK_j)x_k|_{k-d-1}+B(u_{k-d}+K_jx_k|_{k-d-1})||_{P_j^d}\no\\
&-||(\bar A-\bar BK_j)x_k|_{k-d-1}+\bar B(u_{k-d}+K_jx_k|_{k-d-1})||_{P_j^0}]\no\\
=&{\mathsf E}[x_k|_{k-d-1}'K_j'RK_jx_k|_{k-d-1}+x_k'Qx_k\no\\
&-||u_{k-d}||_{B'P_j^dB+\bar B'P_j^0\bar B}+||K_jx_k|_{k-d-1}||_{B'P_j^dB+\bar B'P_j^0\bar B}\no\\
&-2(x_k|_{k-d-1})'(A'P_j^dB+\bar A'P_j^0\bar B)(u_{k-d}+K_jx_k|_{k-d-1})],\label{expeqn}
\end{align}
where the first and second equalities have used \eqref{rksys} and Lyapunov-type equations \eqref{ipki}-\eqref{ipkd},  respectively.

Next, it will be shown that for a given  stabilizing $K_j$, $(P^0_j, \cdots, P_j^d, K_{j+1})$ 
satisfying \eqref{ipki}-\eqref{ki} can be uniquely determined without the knowledge of $\bar A$ and $\bar B$, under certain rank condition. 
 
In fact, \eqref{expeqn} implies the linear equation
\begin{eqnarray}
&&\Theta_j \left[\begin{array}{c}\underline{\mathrm{vec}}(P^0_j) \\ \vdots\\ \underline{\mathrm{vec}}(P^d_j)\\ \mathrm{vec} (B'P_j^dA) \\ \underline{\mathrm{vec}}(B'P_j^dB+\bar B'P_j^0\bar B)\end{array}\right]
=\Gamma_j,\label{veceqn}\\
&&\Theta_j=\begin{bmatrix} z_{d,j}'&z_{d+1,j}'&\cdots& z_{d+l,j}' \end{bmatrix}',\\
&&\Gamma_j=\begin{bmatrix} r_{d,j}&r_{d+1,j}&\cdots& r_{d+l,j} \end{bmatrix}' \label{gammak}
\end{eqnarray}
with
\begin{small}
\begin{eqnarray}
&&z_{k,j}=[\tilde {xx}_{1,j}', \cdots, \tilde {xx}_{d,j}',\hat {xx}_{j}', ux_j^{'}, uu_{j}' ],\\
&&uu_j=\underline{\mathrm{vec}}\left(\mathrm{mat}(u_{k-d,j})-\mathrm{mat}(K_jx_{k,j}|_{k-d-1})\right),\\
&&ux_j=-2\mathrm{vec}(u_{k-d,j} (K_jx_{k,j}|_{k-d-1})'),\\
&& \hat {xx}_j=\underline{\mathrm{vec}}(\mathrm{mat}(x_{k,j}|_{k-d-1})-\mathrm{mat}(x_{k+1,j}|_{k-d})), \\
&&\tilde {xx}_{i,j}=\underline{\mathrm{vec}}(\mathrm{mat}(x_{k,j}|^{k-i}_{k-i-1})-\mathrm{mat}(x_{k+1,j}|^{k+1-i}_{k-i})),\\
&&r_{k,j}=x_{k,j}|_{k-d-1}'K_j'RK_jx_{k,j}|_{k-d-1}+x_{k,j}'Qx_{k,j}.
\end{eqnarray}
\end{small}

In the above, the subscript $j$  indicates that the data is generated by system \eqref{sys} under the controller $-K_jx_k|_{k-d-1}+e_k$,  and
 $x_{k,j}|_{k-i}$ can be represented as 
\begin{eqnarray}
&&x_{k,j}|_{k-i}={\mathcal A}_{i-1} {\mathcal X}_{k-i,j},\\
&&{\mathcal A}_i=[A^{(i)}, A^{(i-1)}B, \cdots, B], \\
&&{\mathcal X}_{k-i,j}=[x_{k-i,j}', u_{k-i-d,j}', \cdots, u_{k-1-d,j}']'.
\end{eqnarray}
It is evident that  ${\mathcal X}_{k-i,j}$ for $i=1,\cdots, d+1$ can be measured indirectly by the history data $x_{k-d,j}, u_{k-1-d,j}$, $\cdots$, $u_{k-2d,j}$ when $(A,B)$ is known but $(\bar A, \bar B)$ unknwon. 

If \eqref{veceqn} has a unique solution of $B'P_j^dB+\bar B'P_j^0\bar B$, $B'P_j^dA+\bar B'P_j^0\bar A$, and $P^i_j$ for $i=0,\cdots, d$,  then $K_{j+1}$ can
be  obtained from
\begin{eqnarray}
K_{j+1}=(R+B'P_j^dB+\bar B'P_j^0\bar B)^{-1}(B'P_j^dA+\bar B'P_j^0\bar A).\label{kk+1}
\end{eqnarray}

Now, we give the  RL-based algorithm 1.

\begin{algorithm}
\caption{RL-based optimal controller design}
\begin{itemize}
\item[1)] Set $j= 0$ and select $K_0$ such that $x_{k+1}=A_kx_k-B_kK_0x_{k-d-1}$ is asymptotically stable in the mean-square sense;
\item[2)]  Apply the control input $u_k = -K_jx_k|_{k-d-1} + e_k$ to system \eqref{sys} on the
time interval $[k_1, k_2]$, and compute $\Theta_j$ and $\Gamma_j$;
\item[3)] Solve \eqref{veceqn} via batch least squares and \eqref{kk+1}. If $|K_{j+1}-K_j| < \epsilon$, where $\epsilon>0$ is a sufficiently small threshold, go
to the next step. Otherwise, set $j +1 \rightarrow j$, and jump 2);
\item[4)] Use $K_j$ as an approximation to the exact control gain $K$ as in \eqref{excgain}.
\end{itemize}
\end{algorithm}
Algorithm 1 is implemented  online  in real time as the data $(x_{k-d}, u_{k-d-1}, \cdots, u_{u-2d})$ is measured  at each time step. Notice that 
$B'P_j^d B+\bar B'P_j^0 \bar B,   P^i_j$ and $B' P_j^d A+\bar B' P_j^0 \bar A$ are $m \times m$, $n \times n$ and $m \times n$ unknown matrices,
 respectively. Particularly, the first two matrices are symmetric. There are actually $l_1\dot= n(n+1)(d+1)/2+m(m+1)/2+mn$ independent elements to be determined in equation \eqref{veceqn}. Therefore,   $l\ge l_1$  sets of data are required before \eqref{veceqn} can be solved.   Since \eqref{veceqn} stems from \eqref{expeqn}, where the equality holds when taking mathematical expectation, we approximate the expectations by numerical average.

\begin{remark} 
Provided that the rank of matrix $\Theta_j$ is kept equal to $l_1$ in the learning process of Algorithm 1, then equation \eqref{veceqn} always has a unique solution. Due to that $P_j^i$ of this solution satisfies the Lyapubov-type equations \eqref{ipki}-\eqref{ipkd} and $K_{j+1}$ is generated by \eqref{kk+1},  according to Theorem 1, the sequences $\{P_j^i\}_{j=0}^{\infty}$ and 
$\{K_j\}_{j=0}^{\infty}$ from solving equation \eqref{veceqn}  converge to the solution  $P^i$ of the Riccati-type equations \eqref{ricpi}-\eqref{ricpd} and the optimal feedback gain $K$ in \eqref{excgain}, respectively.   
\end{remark}

\begin{remark} Denote $l_2\dot=(dm+n)(dm+n+1)/2+m(m+1)/2+m(dm+n)$. 
$l_1$ independent elements are required to be determined in Algorithm 1, while $l_2$  independent elements need to be learned if the Q-learning algorithm is implemented after state augmentation. Given that $l_2-l_1={\mathcal O}(d^2m^2)$,  the computation complexity can be remarkably reduced by using Algoirthm 1 when delay $d$ or  the dimension of the input $m$ are very large. 
\end{remark}


\section{Numerical example}
In this section, a numerical example is provided to evaluate our learning algorithm.
  
Consider system \eqref{sys} and performance index \eqref{costfunj} with parameters
\begin{align}
A&=\begin{bmatrix} 1.1&-0.3\\1 &0\end{bmatrix}, \bar A=\begin{bmatrix} 0 &0\\-0.18 &0\end{bmatrix}, B=\begin{bmatrix} 1 \\0\end{bmatrix},\no\\
 \bar B&=\begin{bmatrix} -0.1 \\0.08\end{bmatrix},  Q=\begin{bmatrix}1 &0.5\\0.5 &1\end{bmatrix}, R=1, d=2.
\end{align}
From \eqref{excgain}, the exact optimal control gain of  the LQR problem is $K^*=[0.8558~-0.2243]$.

We select $K_0=[0~ 0]$ because system \eqref{sys} with  $u_{k-d}=0$ is asymptotically mean-square stable.
In the simulation, the initial data are $x_0=[0.4~0.6]'$, $u_{-2}=-0.2$ and $u_{-1}=-0.45$. From $k=0$ to $k=38$, $400$ scalar Gaussian white noise sequences with zero mean and variance $2.5$ are selected as the exploration noises and used as the system input.

Collect $400$ sets of samples of state and input information over $[0, 40]$ and take their own average. The policy is iterated from $41$, and convergence is attained after $10$ iterations, when the stopping criterion $||K_k-K^*||\le 10^{-4}$ is satisfied. The formulated controller  is used as the actual control input to the system starting from $k=39$ to the end of the simulation. A sample path of the state are ploted in Fig. \ref{trajectories}. 
 
Algorithm 1 gives the control gain matrix $K_9=[0.8626~-0.2151]$. 
 As shown in Fig.\ref{gains},  the convergence of $K_k$ to $K^*$  is illustrated in Fig. \ref{gains}. 

\begin{figure}[h]
\centering
{\includegraphics[scale=0.5]{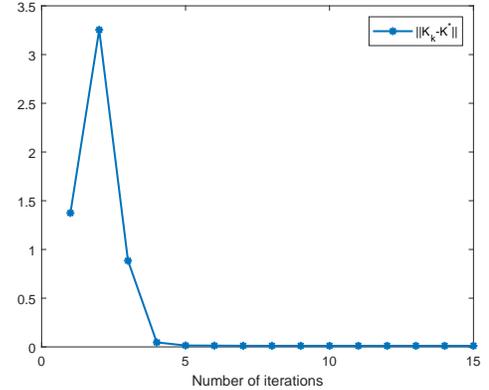}}
\caption{Convergence of $K_k$ to the optimal value of  $K^*$ }\label{gains}
\end{figure}
\begin{figure}[h]
\centering
{\includegraphics[scale=0.5]{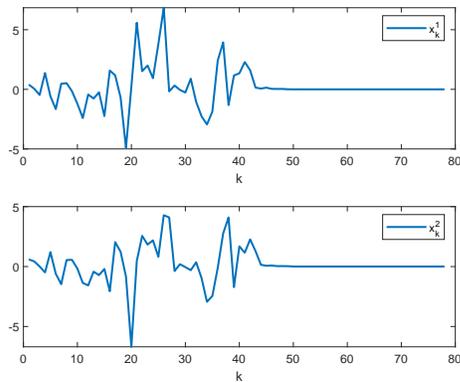}}
\caption{A sample path of the state  during the simulation}\label{trajectories}
\end{figure}

\section{Conclusion}
This paper has obtained the necessary and sufficient stabilizing condition of the predictor-feedback control, which generalizes the classical Lyapunov theory. By applying the condition, two optimal control algorithms for the LQR for multiplicative-noise system with input delay have been proposed. One is model-based and offline, and its convergence and stability analysis have been proved. Another is data-based in the case of the partially unknown
dynamics, and its effectiveness  has also been illustrated by a numerical
example.

\bibliographystyle {plain} 
\bibliography {reference}

\begin{thebibliography}{10}

\bibitem{bianauto2014}
Tao Bian, Yu~Jiang, and Zhong-Ping Jiang.
\newblock Adaptive dynamic programming and optimal control of nonlinear
  nonaffine systems.
\newblock {\em Automatica}, 50:2624–2632, 2014.

\bibitem{biantac2016}
Tao Bian and Zhong-Ping Jiang.
\newblock Adaptive dynamic programming for stochastic systems with state and
  control dependent noise.
\newblock {\em IEEE Transactions on Automatic Control}, 61(12):4170--4175,
  2016.

\bibitem{bianauto2016}
Tao Bian and Zhong-Ping Jiang.
\newblock Value iteration and adaptive dynamic programming for data-driven
  adaptive optimal control design.
\newblock {\em Automatica}, 71:348–360, 2016.

\bibitem{coppens2020ldc}
Peter Coppens, Mathijs Schuurmans, and Panagiotis Patrinos.
\newblock Data-driven distributionally robust {LQR} with multiplicative noise.
\newblock In {\em Proceedings of the 2nd Conference on Learning for Dynamics
  and Control}, volume 120, pages 521--530, 2020.

\bibitem{gravell2021tac}
Benjamin Gravell, Peyman~Mohajerin Esfahani, and Tyler Summers.
\newblock Learning optimal controllers for linear systems with multiplicative
  noise via policy gradient.
\newblock {\em IEEE Transactions on Automatic Control}, 66(11):5283--5298,
  2021.

\bibitem{1971hewer}
Gary~A. Hewer.
\newblock An iterative technique for the computation of the steady state gains
  for the discrete optimal regulator.
\newblock {\em IEEE Transactions on Automatic Control}, 16(4):382--384, 1971.

\bibitem{huang2006wcica}
Yulin Huang, Weihai Zhang, and Huanshui Zhang.
\newblock Infinite horizon {LQ} optimal control for discrete-time stochastic
  systems.
\newblock In {\em 6th World Congress on Intelligent Control and Automation},
  volume~1, pages 252--256, 2006.

\bibitem{huangajc2008}
Yulin Huang, Weihai Zhang, and Huanshui Zhang.
\newblock Infinite horizon linear quadratic optimal control for discrete-time
  stochastic systems.
\newblock {\em Asian Journal of Control}, 10(5):608--615, 2008.

\bibitem{jiang2011nn}
Yu~Jiang and Zhong-Ping Jiang.
\newblock Approximate dynamic programming for optimal stationary control with
  control-dependent noise.
\newblock {\em IEEE Transactions on Neural Networks}, 22(12):2392--2398, 2011.

\bibitem{jiang2018auto}
Yu~Jiang and Zhong-Ping Jiang.
\newblock Computational adaptive optimal control for continuous-time linear
  systems with completely unknown dynamic.
\newblock {\em Automatica}, 48:2699--2704, 2018.

\bibitem{bahareauto2014}
Bahare Kiumarsi, Frank~L. Lewis, Hamidreza Modares, Ali Karimpour, and
  Mohammad-Bagher Naghibi-Sistani.
\newblock Reinforcement q-learning for optimal tracking control of linear
  discrete-time systems with unknown dynamics.
\newblock {\em Automatica}, 50(4):1167--1175, 2014.

\bibitem{leongauto2020}
Alex~S. Leong, Arunselvan Ramaswamy, Daniel~E. Quevedo, Holger Karl, and Ling
  Shi.
\newblock Deep reinforcement learning for wireless sensor scheduling in
  cyber–physical systems.
\newblock {\em Automatica}, 113:108759, 2020.

\bibitem{lewis2011smc}
Frank~L. Lewis and Kyriakos~G. Vamvoudakis.
\newblock Reinforcement learning for partially observable dynamic processes:
  Adaptive dynamic programming using measured output data.
\newblock {\em IEEE Transactions on Systems Man \& Cybernetics Part B
  Cybernetics A Publication of the IEEE Systems Man \& Cybernetics Society},
  41(1):14--25, 2011.

\bibitem{litac2022}
Na~Li, Xun Li, Jing Peng, and Zuo~Quan Xu.
\newblock Stochastic linear quadratic optimal control problem: A reinforcement
  learning method.
\newblock {\em IEEE Transactions on Automatic Control}, 67(9):5009--5016, 2022.

\bibitem{Mnih2015nature}
Volodymyr Mnih, Koray Kavukcuoglu, David Silver, Andrei~A. Rusu, Joel Veness,
  Marc~G. Bellemare, Alex Graves, Martin Riedmiller, Andreas~K. Fidjeland,
  Georg Ostrovski, Stig Petersen, Charles Beattie, Amir Sadik, Ioannis
  Antonoglou, Helen King, Dharshan Kumaran, Daan Wierstra, Shane Legg, and
  Demis Hassabis.
\newblock Human-level control through deep reinforcement learning.
\newblock {\em Nature}, 518:529--533, 2015.

\bibitem{modarestac2014}
Hamidreza Modares and Frank~L. Lewis.
\newblock Linear quadratic tracking control of partially-unknown
  continuous-time systems using reinforcement learning.
\newblock {\em IEEE Transactions on Automatic Control}, 59(11):3051--3056,
  2014.

\bibitem{eric2010irs}
Erik Schuitema, Lucian Busoniu, Robert Babuska, and Pieter Jonker.
\newblock Control delay in reinforcement learning for real-time dynamic
  systems: A memoryless approach.
\newblock In {\em Intelligent Robots and Systems}, pages 3226--3231, 2010.

\bibitem{Silver2018science}
David Silver, Thomas Hubert, Julian Schrittwieser, Ioannis Antonoglou, Matthew
  Lai, Arthur Guez, Marc Lanctot, L.~Sifre, Dharshan Kumaran, Thore Graepel,
  Timothy~P. Lillicrap, Karen Simonyan, and Demis Hassabis.
\newblock A general reinforcement learning algorithm that masters chess, shogi,
  and go through self-play.
\newblock {\em Science}, 362:1140 -- 1144, 2018.

\bibitem{song2010neur}
Ruizhuo Song, Huaguang Zhang, Yanhong Luo, and Qinglai Wei.
\newblock Optimal control laws for time-delay systems with saturating actuators
  based on heuristic dynamic programming.
\newblock {\em Neurocomputing}, 73(16-18):3020--3027, 2010.

\bibitem{suttonmitpress1998}
Richard~S. Sutton and Andrew~G. Barto.
\newblock {\em Reinforcement Learning: An Introduction}.
\newblock Cambridge, MA: MIT Press, 1998.

\bibitem{tanscl2019}
Cheng Tan, Lin Yang, Fangfang Zhang, Zhengqiang Zhang, and Wing~Shing Wong.
\newblock Stabilization of discrete time stochastic system with input delay and
  control dependent noise.
\newblock {\em Systems \& Control Letters}, 123:62--68, 2019.

\bibitem{wangauto2022}
Hongxia Wang, Zhaorong Zhang, and Juanjuan Xu.
\newblock Reinforcement learning for discrete-time systems with input delay and
  input-dependent noise.
\newblock {\em submitted to Automatica}, 2022.

\bibitem{wangneuro2016}
Tao Wang, Huaguang Zhang, and Yanhong Luo.
\newblock Infinite-time stochastic linear quadratic optimal control for unknown
  discrete-time systems using adaptive dynamic programming approach.
\newblock {\em Neurocomputing}, 171(JAN.1):379--386, 2016.

\bibitem{wei2010acta}
Qinglai Wei, Huaguang Zhang, Derong Liu, and Yan Zhao.
\newblock An optimal control scheme for a class of discrete-time nonlinear
  systems with time delays using adaptive dynamic programming.
\newblock {\em Acta Automatica Sinica}, 36(1):121--129, 2010.

\bibitem{xuauto2012}
Hao Xu, S.~Jagannathan, and Frank~L. Lewis.
\newblock Stochastic optimal control of unknown networked control systems in
  the presence of random delays and packet losses.
\newblock {\em Automatica}, 48(6):1017–1030, 2012.

\bibitem{yong1999stochastic}
Jiongmin Yong and Xun~Yu Zhou.
\newblock {\em Stochastic controls: Hamiltonian systems and HJB equations},
  volume~43.
\newblock Springer Science \& Business Media, 1999.

\bibitem{zhang2011tnn}
Huaguang Zhang, Ruizhuo Song, Qinglai Wei, and Tieyan Zhang.
\newblock Optimal tracking control for a class of nonlinear discrete-time
  systems with time delays based on heuristic dynamic programming.
\newblock {\em IEEE Transactions on Neural Networks}, 22(12):1851--1862, 2011.

\bibitem{zhang2015dis}
Huanshui Zhang, Lin Li, Juanjuan Xu, and Minyue Fu.
\newblock Linear quadratic regulation and stabilization of discrete-time
  systems with delay and multiplicative noise.
\newblock {\em IEEE Transactions on Automatic Control}, 60(10):2599--2613,
  2015.

\bibitem{zhang2014neur}
Jilie Zhang, Huaguang Zhang, Yanhong Luo, and Tao Feng.
\newblock Model-free optimal control design for a class of linear discrete-time
  systems with multiple delays using adaptive dynamic programming.
\newblock {\em Neurocomputing}, 135:163--170, 2014.

\end{thebibliography}

\end{document}